\documentclass{amsart}
\usepackage{amsmath}
\usepackage{amssymb}
\usepackage{amsthm}
\usepackage{epsfig}
\usepackage{amsfonts}
\usepackage{amscd}
\usepackage{mathrsfs}
\usepackage{enumerate}
\usepackage{latexsym}

\newtheorem{theorem}{Theorem}[section]
\newtheorem{lemma}[theorem]{Lemma}

\theoremstyle{definition}
\newtheorem{definition}[theorem]{Definition}

\theoremstyle{remark}
\newtheorem{remark}[theorem]{Remark}
\theoremstyle{question}
\newtheorem{question}[theorem]{Question}
\numberwithin{equation}{section}

\begin{document}

\title[One-point connectifications]{One-point connectifications}

\author{M.R. Koushesh}

\address{Department of Mathematical Sciences, Isfahan University of Technology, Isfahan 84156--83111, Iran}
\address{School of Mathematics, Institute for Research in Fundamental Sciences (IPM), P.O. Box: 19395--5746, Tehran, Iran}
\email{koushesh@cc.iut.ac.ir}
\thanks{This research was in part supported by a grant from IPM (No. 92030418).}




\begin{abstract}
A space $Y$ is called an \textit{extension} of a space $X$ if $Y$ contains $X$ as a dense subspace. An extension $Y$ of $X$ is called a \textit{one-point extension} if $Y\setminus X$ is a singleton. Compact extensions are called \textit{compactifications} and connected extensions are called \textit{connectifications}.

It is well known that every locally compact non-compact space has a one-point compactification (known as the \textit{Alexandroff compactification}) obtained by adding a point at infinity. A locally connected disconnected space, however, may fail to have a one-point connectification. It is indeed a long standing question of Alexandroff to characterize spaces which have a one-point connectification. Here we prove that in the class of completely regular spaces, a locally connected space has a one-point connectification if and only if it contains no compact component.
\end{abstract}

\maketitle

\section{Introduction}

Throughout this article by \emph{completely regular} we mean completely regular and Hausdorff (also referred to as Tychonoff).

A space $Y$ is called an \emph{extension} of a space $X$ if $Y$ contains $X$ as a dense subspace. An extension $Y$ of $X$ is called a \emph{one-point extension} if $Y\setminus X$ is a singleton. Compact extensions are called \emph{compactifications} and connected extensions are called \emph{connectifications}.

It is well known that every locally compact non-compact space has a one-point compactification, known as the \emph{Alexandroff compactification}. (See \cite{A}.) A locally connected disconnected space, however, may fail to have a one-point connectification; trivially, any space with a compact open subspace has no Hausdorff connectification. (The lack of compact open subspaces, however, does not guarantee the existence of a connectification; see \cite{ATTW}, \cite{GKLD}, \cite{PW} and \cite{WW}.) There is indeed an old question of Alexandroff of characterizing spaces which have a one-point connectification. This so far has motivated a significant amount of research. The earliest serious work in this direction dates back perhaps to 1945 and is due to B. Knaster in \cite{K}; it presents a characterization of separable metrizable spaces which have a separable metrizable one-point connectification. Knaster's characterization is as follows.

\begin{theorem}[Knaster; \cite{K}]\label{LORE}
Let $X$ be a separable metrizable space. Then $X$ has a separable metrizable one-point connectification if and only if it can be embedded in a connected separable metrizable space as a proper open subspace.
\end{theorem}

More recently, in \cite{ADvM}, M. Abry, J.J. Dijkstra and J. van Mill have given the following alternative characterization of separable metrizable spaces which have a separable metrizable one-point connectification.

\begin{theorem}[Abry, Dijkstra and van Mill; \cite{ADvM}]\label{LOIE}
Let $X$ be a separable metrizable space in which every component is open. Then $X$ has a separable metrizable one-point connectification if and only if $X$ has no compact component.
\end{theorem}

Here, we characterize locally connected completely regular spaces which have a completely regular one-point connectification. Our characterization resembles in form to that of M. Abry, J.J. Dijkstra and J. van Mill stated in Theorem \ref{LOIE} and, as we will now explain, may be viewed as a dual to Alexandroff's characterization of locally compact Hausdorff spaces having a Hausdorff one-point compactification. Observe that locally compact Hausdorff spaces as well as compact Hausdorff spaces are completely regular. The Alexandroff theorem, now reworded, states that a locally compact completely regular space has a completely regular one-point compactification if and only if it is non-compact. Keeping analogy, we prove that a locally connected completely regular space has a completely regular one-point connectification if and only if it contains no compact component. Our method may also be used to give a description of all completely regular one-point connectifications of a locally connected completely regular space with no compact component. Further, for a locally connected completely regular space with no compact component, we give conditions on a topological property ${\mathscr P}$ which guarantee the space to have a completely regular one-point connectification with ${\mathscr P}$, provided that each component of the space has ${\mathscr P}$. This will conclude Section \ref{HGF}. In Section \ref{KJG}, we will be dealing with $T_1$-spaces. Results of this section are dual to those we proved in Section \ref{HGF} rephrased in the context of $T_1$-spaces. In particular, we will prove that a locally connected $T_1$-space has a $T_1$ one-point connectification if it contains no compact component.

We will use some basic facts from the theory of the Stone--\v{C}ech compactification. Recall that the \emph{Stone--\v{C}ech compactification} of a completely regular space $X$, denoted by $\beta X$, is the Hausdorff compactification of $X$ which is characterized among all Hausdorff  compactifications of $X$ by the fact that every continuous mapping $f:X\rightarrow [0,1]$ is continuously extendable over $\beta X$. The Stone--\v{C}ech compactification of a completely regular space always exists. We will use the following standard properties of $\beta X$. (By a \emph{clopen} subspace we mean a simultaneously closed and open subspace.)
\begin{itemize}
  \item A clopen subspace of $X$ has open closure in $\beta X$.
  \item Disjoint zero-sets in $X$ have disjoint closures in $\beta X$.
  \item $\beta T=\beta X$ whenever $X\subseteq T\subseteq\beta X$.
\end{itemize}
For more information on the subject and other background material we refer the reader to the texts \cite{GJ} and \cite{PW1}.

\section{One-point connectifications of completely regular spaces}\label{HGF}

The following subspace of $\beta X$ plays a crucial role throughout our whole discussion.

\begin{definition}\label{JAS}
Let $X$ be a completely regular space. Define
\[\delta X=\bigcup\{\mathrm{cl}_{\beta X}C:C\mbox{ is a component of }X \},\]
considered as a subspace of $\beta X$.
\end{definition}

Recall that a space $X$ is called \emph{locally connected} if for every $x$ in $X$, every neighborhood of $x$ in $X$ contains a connected neighborhood of $x$ in $X$. Every component of a locally connected space is open and thus is clopen, as components are always closed. Observe that any clopen subspace of a completely regular space $X$ has open closure in $\beta X$. Therefore, in a locally connected completely regular space $X$ each component of $X$ has open closure in $\beta X$, in particular, $\delta X$ is open in $\beta X$.

The following theorem characterizes locally connected completely regular spaces which have a completely regular one-point connectification.

\begin{theorem}\label{LRE}
A locally connected completely regular space has a completely regular one-point connectification if and only if it contains no compact component.
\end{theorem}

\begin{proof}
Let $X$ be a (non-empty) locally connected completely regular space.

\textit{Sufficiency.} Suppose that $X$ contains no compact component. We show that $X$ has a completely regular one-point connectification. Let $C$ be a component of $X$. Then $\mathrm{cl}_{\beta X}C\setminus X$ is non-empty, as $C$ is non-compact. Choose an element $t_C$ in $\mathrm{cl}_{\beta X}C\setminus X$. Let
\[P=\{t_C:C\mbox{ is a component of }X \}\cup(\beta X\setminus\delta X).\]
Note that $P$ misses $X$, as $\beta X\setminus\delta X$ does so, since $X$ is contained in $\delta X$ trivially. Also, $P$ is non-empty, as $X$ is so. We show that $P$ is closed in $\beta X$. Let $t$ be in $\mathrm{cl}_{\beta X}P$. Obviously, $t$ is contained in $P$ if it is contained in $\beta X\setminus\delta X$. Let $t$ be in $\delta X$. Then $t$ is contained in $\mathrm{cl}_{\beta X}D$ for some component $D$ of $X$. We show that $t$ is identical to $t_D$. Suppose otherwise. Then
\[U=\mathrm{cl}_{\beta X}D\setminus\{t_D\}\]
is an open neighborhood of $t$ in $\beta X$. (Observe that the closure in $\beta X$ of $D$ is open in $\beta X$, as $D$ is a component of $X$ and $X$ is locally connected.) We show that $U$ misses $P$. Let $E$ be a component of $X$ distinct from $D$. Then $E$ is necessarily disjoint from $D$. This implies that $E$ and $D$ have disjoint closures in $\beta X$, as they are disjoint zero-sets (indeed, disjoint clopen subspaces) in $X$. Therefore $t_E$ is not in $U$, as it is contained in $\mathrm{cl}_{\beta X}E$. It is trivial that $U$ misses $\beta X\setminus\delta X$. Thus $U$ misses $P$, which is a contradiction. This shows that $P$ is closed in $\beta X$.

Let $T$ be the space which is obtained from $\beta X$ by contracting the compact subspace $P$ of $\beta X$ to a point $p$ and let $\phi:\beta X\rightarrow T$ denote the corresponding quotient mapping. Consider the subspace $Y=X\cup\{p\}$ of $T$. Then $Y$ is completely regular, as $T$ is so, and contains $X$ densely, as $T$ does so. That is, $Y$ is a completely regular one-point extension of $X$. We verify that $Y$ is connected. Note that $p$ is contained in $\mathrm{cl}_Y C$ for every component $C$ of $X$, as
\[p=\phi(t_C)\in\phi(\mathrm{cl}_{\beta X}C)\subseteq\mathrm{cl}_T\phi(C)=\mathrm{cl}_TC.\]
Since $\mathrm{cl}_Y C$ is the closure of a connected space, it is connected for every component $C$ of $X$. Therefore
\[Y=\bigcup\{\mathrm{cl}_Y C:C\mbox{ is a component of }X\}\]
is connected, as it is the union of a collection of connected subspaces of $Y$ with non-empty intersection.

\textit{Necessity.} Suppose that $X$ has a completely regular one-point connectification $Y$. We show that no component of $X$ is compact. Suppose otherwise. Then $X$ contains a compact component $C$. Trivially, $C$ is closed in $Y$, as $Y$ is Hausdorff. On the other hand $C$ is open in $Y$, as $C$ is open in $X$, since $X$ is locally connected, and $X$ is open in $Y$. That is, $C$ is clopen in $Y$. Since $Y$ is connected we then have $C=Y$, which is a contradiction.
\end{proof}

\begin{remark}\label{JKAS}
Theorem \ref{LRE} is valid if we replace locally connectedness of $X$ by the requirement that every component of $X$ is open; this follows trivially by an inspection of the proof.
\end{remark}

The method used in the proof of Theorem \ref{LRE} can be modified to give a description of all completely regular one-point connectifications of a locally connected completely regular space $X$ with no compact component; this will be the context of our next theorem.

The following lemma follows from a very standard argument; we therefore omit the proof.

\begin{lemma}\label{JGFD}
Let $Y=X\cup\{p\}$ be a completely regular one-point extension of a space $X$. Let $\phi:\beta X\rightarrow\beta Y$ be the continuous extension of the identity mapping on $X$. Then $\beta Y$ is the quotient space obtained from $\beta X$ by contracting $\phi^{-1}(p)$ to $p$ and $\phi$ is its quotient mapping.
\end{lemma}

The following describes for a locally connected completely regular space which has no compact component, all its completely regular one-point connectifications.

\begin{theorem}\label{JGF}
Let $X$ be a locally connected completely regular space with no compact component. Let $Y=X\cup\{p\}$ be the quotient space obtained by contracting a non-empty compact subspace of $\beta X\setminus X$ which intersects the closure in $\beta X$ of each component of $X$ to the point $p$. Then $Y$ is a completely regular one-point connectification of $X$. Further, any completely regular one-point connectification of $X$ is obtained in this way.
\end{theorem}

\begin{proof}
Let $P$ be a non-empty compact subspace of $\beta X\setminus X$ which intersects the closure in $\beta X$ of every component of $X$. Let $T$ be the quotient space of $\beta X$ which is obtained by contracting $P$ to a point $p$. An argument similar to the one given in the proof of Theorem \ref{LRE} shows that the subspace $Y=X\cup\{p\}$ of $T$ is a completely regular one-point connectification of $X$.

To show the converse, let $Y=X\cup\{p\}$ be a completely regular one-point connectification of $X$. Let $\phi:\beta X\rightarrow\beta Y$ be the continuous extension of the identity mapping on $X$. It follows from Lemma \ref{JGFD} that $\beta Y$ is the quotient space obtained from $\beta X$ by contracting $\phi^{-1}(p)$ to $p$ and $\phi$ is its quotient mapping. We need to show that $\phi^{-1}(p)$ intersects the closure in $\beta X$ of each component of $X$. Let $C$ be a component of $X$. Suppose to the contrary that
\[\phi^{-1}(p)\cap\mathrm{cl}_{\beta X}C=\emptyset.\]
Then $p$ is not contained in $\phi(\mathrm{cl}_{\beta X}C)$, and since $\mathrm{cl}_{\beta Y}C\subseteq\phi(\mathrm{cl}_{\beta X}C)$, then $p$ is not contained in $\mathrm{cl}_YC$ either. Therefore $C$ is closed in $Y$, as it is closed in $X$. But $C$ is also open in $Y$, as it is open in $X$, since $X$ is locally connected (and $X$ is open in $Y$). This contradicts the connectedness of $Y$.
\end{proof}

In \cite{Kou} (also \cite{Kou2} and \cite{Kou5}--\cite{Kou10}) we have studied topological properties ${\mathscr P}$ such that any completely regular space which has ${\mathscr P}$ locally has a completely regular one-point extension which has ${\mathscr P}$. Motivated by this, we consider conditions on a topological property ${\mathscr P}$ which guarantee a locally connected completely regular space with no compact component to have a completely regular one-point connectification with ${\mathscr P}$, provided that all its components have ${\mathscr P}$.

We need the following definition.

\begin{definition}\label{KALS}
Let ${\mathscr P}$ be a topological property. Then
\begin{enumerate}
  \item ${\mathscr P}$ is \emph{closed hereditary} if any closed subspace of a space having ${\mathscr P}$, has ${\mathscr P}$.
  \item ${\mathscr P}$ is \emph{finitely additive} if any space which is expressible as a finite disjoint union of closed subspaces each having ${\mathscr P}$, has ${\mathscr P}$.
  \item ${\mathscr P}$ is \emph{co-local} if a space $X$ has ${\mathscr P}$ provided that it contains a point $p$ with an open base ${\mathscr B}$ for $X$ at $p$ such that $X\setminus B$ has ${\mathscr P}$ for any $B$ in ${\mathscr B}$.
\end{enumerate}
\end{definition}

\begin{remark}\label{ZKS}
The condition stated in (3) in Definition \ref{KALS} has been introduced by S. Mr\'{o}wka in \cite{Mr}, where it was called \emph{condition $(\mathrm{W})$}.
\end{remark}

\begin{remark}\label{KHG}
Some authors call a topological property ${\mathscr P}$ \emph{finitely additive} if any space which is a finite (and not necessarily disjoint) union of closed subspaces each having ${\mathscr P}$, has ${\mathscr P}$. The reader is warned of the difference between this definition and the definition given in Definition \ref{KALS}.
\end{remark}

\begin{theorem}\label{RFS}
Let $X$ be a locally connected completely regular space with no compact component. Let ${\mathscr P}$ be a closed hereditary finitely additive co-local topological property. If every component of $X$ has ${\mathscr P}$ (in particular, if $X$ has ${\mathscr P}$) then $X$ has a completely regular one-point connectification with ${\mathscr P}$.
\end{theorem}

\begin{proof}
Note that if $X$ has ${\mathscr P}$, then each of its components has ${\mathscr P}$, as ${\mathscr P}$ is closed hereditary. We may therefore prove the theorem in the case when every component of $X$ has ${\mathscr P}$.

Let $P$, $T$, $\phi$ and $Y$ be as defined in the proof of Theorem \ref{LRE}. Since ${\mathscr P}$ is co-local, to show that $Y$ has ${\mathscr P}$ it suffices to show that $Y\setminus U$ has ${\mathscr P}$ for any open neighborhood $U$ of $p$ in $Y$. Let $U$ be an open neighborhood of $p$ in $Y$ and let $U'$ be an open subspace of $T$ with $U=U'\cap Y$. Then
\[\beta X\setminus\delta X\subseteq\phi^{-1}(p)\subseteq\phi^{-1}(U'),\]
as $p$ is contained in $U'$, and thus
\[\beta X\setminus\phi^{-1}(U')\subseteq\delta X.\]
By compactness (and the definition of $\delta X$) it then follows that
\begin{equation}\label{JOB}
\beta X\setminus\phi^{-1}(U')\subseteq\mathrm{cl}_{\beta X}C_1\cup\cdots\cup\mathrm{cl}_{\beta X}C_n,
\end{equation}
where $C_i$ is a component of $X$ for each $i=1,\ldots,n$. Intersecting both sides of (\ref{JOB}) with $X$ gives
\[X\setminus U\subseteq C_1\cup\cdots\cup C_n=D.\]
Note that $D$ has ${\mathscr P}$, as it is a finite disjoint union of closed subspaces each with ${\mathscr P}$ and ${\mathscr P}$ is finitely additive. Thus
\[Y\setminus U=X\setminus U\]
has ${\mathscr P}$, as it is closed in $D$ and ${\mathscr P}$ is closed hereditary.
\end{proof}

\begin{remark}\label{FDF}
There is a long list of topological properties, mostly covering properties (topological properties described in terms of the existence of certain kinds of open subcovers or refinements of a given open cover of a certain type), satisfying the requirements of Theorem \ref{RFS}. Specifically, we mention the Lindel\"{o}f property, paracompactness, metacompactness, subparacompactness, the para-Lindel\"{o}f property, the $\sigma$-para-Lindel\"{o}f  property, weak $\theta$-refinability, $\theta$-refinability (or  submetacompactness), weak $\delta\theta$-refinability, and $\delta\theta$-refinability (or the submeta-Lindel\"{o}f property). (See Example 2.16 in \cite{Kou} for the proof and see \cite{Bu}, \cite{Steph} and \cite{Va} for the definitions.)
\end{remark}

\section{One-point connectifications of $T_1$-spaces}\label{KJG}

This section deals with one-point connectifications of $T_1$-spaces. The results of this section will be dual to those we have obtained in the previous section. We will make critical use of the Wallman compactification; this will replace the Stone--\v{C}ech compactification, as used in the previous section.

Recall that the \emph{Wallman compactification} of a $T_1$-space $X$, denoted by $wX$, is the $T_1$ compactification of $X$ with the property that every continuous mapping $f:X\rightarrow K$ of $X$ to a compact Hausdorff space $K$ is continuously extendable over $wX$. The Wallman compactification is the substitute of the Stone--\v{C}ech compactification which is defined for every $T_1$-space. The Wallman compactification of a $T_1$-space $X$ is Hausdorff if and only if $X$ is normal, and in this case, it coincides with the Stone--\v{C}ech compactification of $X$. The Wallman compactification has properties which are dual to those of the Stone--\v{C}ech compactification. In particular, a clopen subspace of a $T_1$-space $X$ has open closure in $wX$, and disjoint zero-sets in $X$ have disjoint closures in $wX$.

The next theorem is dual to Theorem \ref{LRE}.

\begin{theorem}\label{JKHG}
A locally connected $T_1$-space has a $T_1$ one-point connectification if it contains no compact component.
\end{theorem}

\begin{proof}
Let $X$ be a (non-empty) locally connected $T_1$-space with no compact component. Let $\delta X$, $t_C$ and $P$ be as defined in (Definition \ref{JAS} and) the proof of Theorem \ref{LRE} with $\beta X$ substituted by $wX$ in their definitions. As argued in the proof of Theorem \ref{LRE} it follows that $P$ is a non-empty closed subspace of $wX$ which misses $X$. Let $T$ be the quotient space of $wX$ which is obtained by contracting $P$ to a point $p$. Then $T$ is a $T_1$-space, as singletons are all closed in $T$. As argued in the proof of Theorem \ref{LRE} the subspace $Y=X\cup\{p\}$ of $T$ is a connected $T_1$ one-point extension of $X$.
\end{proof}

We do not know whether the converse of Theorem \ref{JKHG} holds true; we state this formally as an open question.

\begin{question}\label{KHF}
For a locally connected $T_1$-space does the existence of a $T_1$ one-point connectification imply the non-existence of compact components?
\end{question}

The next two theorems are dual to Theorems \ref{JGF} and \ref{RFS}, respectively. We omit the proofs, as they are analogous to the proofs we have already given for Theorems \ref{JGF} and \ref{RFS}, respectively (with the use of Theorem \ref{JKHG} in place of that of Theorem \ref{LRE}).

\begin{theorem}\label{JHFDF}
Let $X$ be a locally connected $T_1$-space with no compact component. Let $Y=X\cup\{p\}$ be the quotient space obtained by contracting a non-empty closed subspace of $wX$, which is contained in $wX\setminus X$ and intersects the closure in $wX$ of each component of $X$, to the point $p$. Then $Y$ is a $T_1$ one-point connectification of $X$.
\end{theorem}

\begin{question}\label{HJGDF}
For a locally connected $T_1$-space $X$ with no compact component, does Theorem \ref{JHFDF} give every $T_1$ one-point connectification of $X$?
\end{question}

\begin{theorem}\label{HJGFD}
Let $X$ be a locally connected $T_1$-space with no compact component. Let ${\mathscr P}$ be a closed hereditary finitely additive co-local topological property. If every component of $X$ has ${\mathscr P}$ (in particular, if $X$ has ${\mathscr P}$) then $X$ has a $T_1$ one-point connectification with ${\mathscr P}$.
\end{theorem}

\end{document}